\documentclass{amsart}
\usepackage{amssymb}
\usepackage{pb-diagram}
\usepackage{color}
\usepackage[normalem]{ulem}

\newtheorem{theorem}{Theorem}[section]

\newtheorem{prop}[theorem]{Proposition}

\newtheorem{claim}[theorem]{Claim}
\newtheorem{fact}[theorem]{Fact}
\newtheorem{cor}[theorem]{Corollary}

\newtheorem{lemma}[theorem]{Lemma}

\newtheorem{question}[theorem]{Question}

\theoremstyle{definition}
\newtheorem{defn}[theorem]{Definition}

\newtheorem{example}[theorem]{Example}

\theoremstyle{remark}

\newtheorem{remark}[theorem]{Remark}

\newcommand{\WM}{\widetilde{\cal M}}

\newcommand{\la}{\langle}
\newcommand{\ra}{\rangle}
\newcommand{\CM}{{\cal M}}

\newcommand{\sub}{\subseteq}
\newcommand{\elsub}{\preccurlyeq}

\newcommand{\dcl}{\operatorname{dcl}}

\newcommand{\cl}{\operatorname{cl}}

\newcommand{\rk}{\ensuremath{\textup{rk}}}

\newcommand{\cal}[1]{\ensuremath{\mathcal{#1}}}
\newcommand{\Cal}[1]{\ensuremath{\mathcal{#1}}}

\newcommand{\Lrarr}{\ensuremath{\Leftrightarrow}}
\newcommand{\Rarr}{\ensuremath{\Rightarrow}}

\newcommand{\res}{\ensuremath{\upharpoonright}}

\newcommand{\es}{\ensuremath{\emptyset}}

\newcommand{\Aut}{\ensuremath{\textup{Aut}}}

\newcommand{\sm}{\setminus}

\newcommand{\Z}{\mathbb{Z}}
\newcommand{\N}{\mathbb{N}}
\newcommand{\Q}{\mathbb{Q}}
\newcommand{\R}{\mathbb{R}}

\title[Small sets in dense pairs]
{Small sets in dense pairs}


\date{\today}

\begin{document}

\author {Pantelis  E. Eleftheriou}

\address{Department of Mathematics and Statistics, University of Konstanz, Box 216, 78457 Konstanz, Germany}

\email{panteleimon.eleftheriou@uni-konstanz.de}

\thanks{Research supported by an Independent Research Grant from the German Research Foundation (DFG) and a Zukunftskolleg Research Fellowship.}

\begin{abstract}
Let $\widetilde{\cal M}=\la \cal M, P\ra$ be an expansion of an  o-minimal structure $\cal M$ by a dense set $P\sub M$, such that three tameness conditions hold. We prove that the induced structure on $P$ by $\cal M$ eliminates imaginaries. As a corollary, we obtain that every small set $X$ definable in  $\WM$ can be definably embedded into some $P^l$, uniformly in parameters, settling a question from \cite{egh}. We  verify the tameness conditions in three examples: dense pairs of real closed fields, expansions of $\cal M$ by a dense independent set, and expansions by a dense divisible multiplicative  group with the Mann property. Along the way, we point out a gap in the proof of a relevant  elimination of imaginaries result  in Wencel \cite{wen}. The above results are in contrast to  recent literature, as it is known  in general that $\WM$ does not eliminate imaginaries, and neither it nor the induced structure on $P$ admits definable Skolem functions.
 \end{abstract}
 \maketitle

\section{Introduction}
Elimination of imaginaries is a classical theme in model theory. If a structure eliminates imaginaries, then quotients of definable sets by definable equivalent relations can be treated as definable.

\begin{defn}\label{ei} A structure $\cal N$ \emph{eliminates imaginaries}  if for every $\emptyset$-definable equivalence relation $E$ on $N^n$, there is a $\emptyset$-definable map $f:N^n\to N^l$ such that for every $x, y\in N^n$,
$$E(x,y) \,\,\Lrarr\,\,f(x)=f(y).$$
In particular, $N^n/E$ is in bijection with the $\es$-definable set $\{f(a): a\in N^n\}$.
\end{defn}

We fix throughout this paper an o-minimal expansion $\cal M=\la M, <, +, 0, \dots\ra$ of an ordered group with a distinguished positive element $1$. We denote by $\cal L$ its language, and by $\dcl$ the usual definable closure operator in \cal M.  An `$\cal L$-definable' set is a set definable in $\cal M$ with parameters. We write `$\cal L_A$-definable' to specify that those parameters come from $A\sub M$.  It is a well-known fact that $\cal M$ admits definable Skolem functions and eliminates imaginaries (\cite[Chapter 6]{vdd-book}).

\begin{defn} Let $D, P\sub M$. The \emph{$D$-induced structure on $P$ by \cal M}, denoted by $P_{ind(D)}$, is a structure in the  language
$$\cal L_{ind(D)}=\{R_{\phi(x)}(x): \phi(x)\in \cal L_D\},$$
whose universe is $P$ and, for every tuple $a\sub P$,
 $$P_{ind(D)}\models R_\phi(a) \,\,\Lrarr\,\, \cal M\models \phi(a).$$
\end{defn}
For an extensive account on $P_{ind(D)}$, see \cite[Section 3.1.2]{simon-book}. \\

\textbf{For the rest of this paper we fix some $P\sub M$ and denote $\widetilde{\cal M}=\la \cal M, P\ra$.}  We let $\cal L(P)$ denote the language of $\WM$; namely, the language $\cal L$ augmented by a unary predicate symbol $P$. We denote by $\dcl_{\cal L(P)}$ the definable closure operator in $\WM$.  Unless stated otherwise, by `($A$-)definable' we mean ($A$-)definable in $\widetilde{\cal M}$, where $A\sub M$. We  use the letter $D$ to denote an arbitrary, but not fixed, subset of $M$.

Consider the following three properties for $\widetilde{\cal M}$ and $D$:

\begin{itemize}
  \item[(OP)] (Open definable sets are $\cal L$-definable.) For every set $A$ such that $A\setminus P$ is $\dcl$-independent over $P$, and for every $A$-definable set $V \subset M^n$, its topological closure $\overline V \subseteq M^{n}$ is $\cal L_A$-definable.\smallskip

  \item[(dcl)$_D$] Let $B, C\sub P$ and
$$A=\dcl(BD)\cap \dcl(CD)\cap P.$$
Then $$\dcl(AD)=\dcl(BD)\cap \dcl(CD).$$

\item[(ind)$_D$] Every $A$-definable set in $P_{ind(D)}$ is the trace of an $\cal L_{AD}$-definable set.
\end{itemize}

Properties (OP) and (ind)$_D$ already appear in the literature and are  known to hold for the examples  mentioned in Theorem C below (Fact \ref{inparticular}). Property (OP) is Assumption (III) from \cite{egh}, and a justification for its terminology is provided in \cite[Lemma 2.5]{egh}.
Property (dcl)$_D$ is introduced here and it is established for the examples of Theorem C  in Section \ref{sec-examples}.
The thrust of (dcl)$_D$ is that it concerns only definability in $\cal M$.

Our first result is the following:\vskip.3cm

\noindent\textbf{Theorem A.} {\em Assume \textup{(OP)}, \textup{(dcl)}$_D$ and \textup{(ind)$_D$}  hold for $\widetilde{\cal M}$ and  $D$, and that $D$ is $\dcl$-independent over $P$. Then $P_{ind(D)}$ eliminates imaginaries.}\vskip.3cm

Theorem A stands in contrast to the general intuition that in pairs with tame geometric behavior on the class of all definable sets,  `choice properties'  generally fail. Such pairs have been extensively studied in the literature, with the primary example being that of  a real closed field expanded by a dense real closed subfield, considered by A. Robinson \cite{rob}.  Further examples include arbitrary dense pairs \cite{vdd-dense} of o-minimal structures, and expansions of \cal M by a dense independent set \cite{dms2} or by a dense multiplicative group with the Mann Property \cite{dg}. It is known that a dense pair does not eliminate imaginaries, and neither it nor $P_{ind(D)}$ admits definable Skolem functions (\cite[Section 5]{dms} and \cite{ehk}). If $P$ is a dense independent set, then $\widetilde{\cal M}$ eliminates imaginaries but does not admit definable Skolem functions (\cite{dms2}).

In \cite{egh}, the above and further examples were all put under a common perspective and a program was initiated for understanding their definable sets in terms of $\cal L$-definable sets and `$P$-bound' sets. In particular, they were shown to satisfy (OP). An important application of Theorem A  is (Theorem B below)  that the study of $P$-bound sets can be further reduced to that of definable (in $\WM$) subsets of $P^l$. For the examples of Theorem C, those are definable in the induced structure on $P$, by \cite{dms2,vdd-dense, dg}.
This reduction is  the main motivation of the present work.

\begin{defn}[\cite{vdd-dense}]
A set $X\sub M^n$ is called \emph{$P$-bound over $A$} if there is an $\cal L_A$-definable function $h: M^m \to M^n$ such that $X \subseteq h(P^m)$.
\end{defn}

In the aforementioned examples, $P$-boundness amounts to a precise topological notion of smallness (\cite[Definition 2.1]{egh}), as well as to the classical notion of $P$-internality from geometric stability theory (\cite[Corollary 3.12]{egh}). Our application of Theorem A is the following.\vskip.3cm

\noindent\textbf{Theorem B.} {\em Assume \textup{(OP)}, \textup{(dcl)}$_D$ and \textup{(ind)$_D$} hold for $\widetilde{\cal M}$ and  $D$, and that $D$ is $\dcl$-independent over $P$. Let  $X\sub M^n$ be a $D$-definable set. If $X$ is $P$-bound over $D$, then there is a $D$-definable injective map $\tau:X\to P^l$.}\vskip.3cm

The core of this paper is  Section \ref{sec-examples}, where we establish (dcl)$_D$ in three main examples of tame expansions of o-minimal structures by dense sets, under an additional assumption on  $D$.\\

\noindent\textbf{Theorem C.} {\em Suppose $\WM=\la \cal M, P\ra$ is one of the following structures:
\begin{itemize}
  \item[(a)] a dense pair of real closed fields; that is, $\cal M$ is a real closed field and $P$  a dense elementary substructure of $\cal M$,
  \item[(b)] an expansion of $\cal M$ by a dense $\dcl$-independent set $P$,
  \item[(c)] an expansion of a real closed field $\cal M$ by a dense divisible subgroup $P$ of $\la M^{>0}, \cdot\ra$ with the Mann property.
      \end{itemize}
Suppose $D\sub M$ is $\dcl$-independent over $P$. Then \textup{(dcl)$_D$} holds. In particular, $P_{ind(D)}$ eliminates imaginaries.}\\

We show in Example \ref{noEI} that the assumption of $D$ being $\dcl$-independent over $P$ is necessary; namely, without it, $P_{ind(D)}$ need not eliminate imaginaries.
However, even for arbitrary $A$-definable sets, one still obtains the following corollary, which in particular applies to our examples.

\begin{cor}\label{corD} Assume \textup{(OP)}, \textup{(dcl)}$_D$ and \textup{(ind)$_D$} hold for $\widetilde{\cal M}$ and every $D\sub M$ which is $\dcl$-independent over $P$. Let $X\sub M^n$ be an $A$-definable set. If $X$ is $P$-bound over $A$,  then there is an $A\cup P$-definable injective map $\tau:X\to P^l$.
\end{cor}

Allowing parameters from $P$ is standard practice when studying definability in this context; see for example also \cite[Lemma 2.5 and Corollary 3.26]{egh}.

We note that Corollary \ref{corD} settles affirmatively \cite[Question 7.12]{egh} in our examples. The same question was asked to the author by E. Baro and A. Martin-Pizarro during the Summer School in Tame Geometry in Konstanz in 2016.

In Section \ref{sec-optimality}, we establish the optimality of our results. Besides the aforementioned Example \ref{noEI}, we prove in Corollary \ref{dclnec} that (dcl)$_D$ is  necessary for $P_{ind(D)}$ to eliminate imaginaries. More precisely, we introduce (earlier, in Section \ref{sec-examples}) a further property for $D$, called ($\dcl'$)$_D$, and show in Proposition \ref{dclnec2}  that if (OP) and (ind)$_D$ hold, and $D$ is $\dcl$-independent over $P$, then
\begin{equation}
\text{ $P_{ind(D)}$ eliminates imaginaries} \,\,\, \Lrarr\,\,\,\, \text{(dcl)}_D \,\,\,\,\Lrarr\,\,\,\, \text{($\dcl'$)}_D.\notag
 \end{equation}
In Example \ref{dclDfail} we observe that if we do not assume (OP), the above three properties need not hold. We do not know whether they hold, if we assume (OP) and (ind)$_D$ (but not that $D$ is $\dcl$-independent over $P$). Finally, (OP) does not imply (ind)$_D$ (Remark \ref{rmk-P}), but we do not know whether (ind)$_D$ is necessary for $P_{ind(D)}$ to eliminate imaginaries (Question \ref{qn-mann}).

In the Appendix, written jointly with P. Hieronymi, we show that Theorem C does not hold for arbitrary dense pairs of o-minimal structures. Namely, we provide an example of an o-minimal trace that does not eliminate imaginaries.

 Our proof of Theorem A is influenced by two previous accounts on elimination of imaginaries in ordered structures, \cite{pi} and \cite{wen}, but diverges from both of them substantially. As noted in Fact \ref{fact-ei2} below,
to prove that an ordered pregeometric structure \cal N eliminates imaginaries, the following is enough:
 \begin{itemize}
  \item[(*)]  Let $B, C\sub N$ and $A=\dcl_{\cal N}(B) \cap \dcl_{\cal N}(C)$. If $X\sub N^n$ is $B$-definable and $C$-definable, then $X$ is $A$-definable.
\end{itemize}
So our approach to Theorem A is to prove (*) for $\cal N=P_{ind(D)}$  (Lemma \ref{pillay2}). There are two obstacles in adopting the existing accounts. First, Pillay \cite{pi} establishes (*) for an o-minimal $\cal N$ under the additional assumption that $A\elsub \cal N$, which need no longer be true here (\cite[Proposition 2.4]{ehk}). Second, $P_{ind(D)}$ is only weakly o-minimal, and Wencel's proof \cite{wen} of elimination of imaginaries in certain weakly o-minimal structures (even under the assumption $A\elsub P_{ind(D)}$) unfortunately contains a serious gap (see Section \ref{sec-wen}). We are thus led to produce a new strategy for proving (*) that moreover works under our general assumptions.

$ $\\
\noindent\emph{Structure of the paper.} In Section \ref{prelim}, we recall some basic facts about elimination of imaginaries and discuss the pregeometry in $P_{ind(D)}$. Section \ref{sec-main} contains the proofs of Theorems A and B. We also point out the aforementioned gap in \cite{wen}. A major part of our work is
Section \ref{sec-examples}, where we prove (dcl)$_D$ in the three main examples: dense pairs of real closed fields, expansions of $\cal M$ by a dense independent set, and expansions by a dense multiplicative group with the Mann property.  In Section \ref{sec-optimality}, we provide examples to establish the optimality of our results.

$ $\\
\noindent\emph{Acknowledgements.} The author wishes to thank Ayhan G\"unaydin, Assaf Hasson, and Philipp Hieronymi for many helpful discussions on this topic, and for pointing out  corrections to the first draft. Thanks also to Alex Savatovsky and Patrick Speissegger   for providing their important feedback during  a seminar series at the University of Konstanz. Finally, thanks to the referee for many helpful comments that improved the original manuscript.

\section{Preliminaries}\label{prelim}
We assume familiarity with the basics of o-minimality and pregeometries, as  can be found, for example, in \cite{vdd-book} or \cite{pi}. A tuple of elements is denoted just by one element, and we write $b\sub B$ if $b$ is a tuple with coordinates from $B$. The set of realizations of a formula $\phi(x)$ in a structure $\cal N$ is denoted by $\phi(N^n)$, where $x$ is an $n$-tuple. If $A\sub N$, we write $A \elsub \cal N$ to denote that $A$ is an elementary substructure of $\cal N$ in the language of $\cal N$.  If $A, B\sub N$, we often write $AB$ for $A\cup B$. We denote by $\Aut(\cal N)$ the group of automorphisms of $\cal N$. We denote by $\Gamma(f)$ the graph of a function $f$.

Recall that  $\cal M=\la M, <, +, 0, \dots\ra$ is our fixed o-minimal expansion of an ordered group with a distinguished positive element $1$. We denote the definable closure operator in \cal M by $\dcl$, the corresponding rank of a tuple by $\rk$, and the dimension of an $\cal L$-definable set by $\dim$. The topological closure of a set $X\sub M^n$ is denoted by $\overline X$.
If $X, Y\sub M^n$, we call $X$ dense in $Y$, if $\overline{X\cap Y}=\overline Y$.

\subsection{Elimination of imaginaries}
The property of elimination of imaginaries (Definition \ref{ei}) can be formulated in many  ways. In Fact \ref{fact-ei} we  state  one  which will be useful for our purposes. Notice that Definition \ref{ei} is sometimes called \emph{uniform} elimination of imaginaries,
whereas elimination of imaginaries is reserved for the condition that every definable set $X$ has a canonical parameter (see below).
However, in the presence of two distinct constants in our $\cal L$, the two notions coincide (\cite[Lemma 4.3]{ma}).

For the rest of this subsection,  let $\cal N$  be a sufficiently saturated structure with two distinct constants in its language, and $X\sub N^n$ a definable set. We call $A\sub N$ a \emph{defining set for $X$}, if $X$ is $A$-definable. The following fact is noted in \cite[Section 3]{pi}.

\begin{fact}\label{fact-ei} Assume $\cal N$ is an ordered structure. Then $\cal N$ eliminates imaginaries if and only if every definable set has a smallest definably closed defining set. \end{fact}

It is  well-known that if $C$ is the smallest definably closed defining set for $X$, then $C=\dcl_{\cal N}(p)$, for some tuple $p\sub N$ satisfying: for every $\tau\in \Aut(\cal N)$,
$$\tau(p)=p \,\,\,\,\Lrarr\,\,\,\, \tau(X)=X.$$
Such a tuple $p$  is called a \emph{canonical parameter}  for $X$.  Clearly, if $X$ is $A$-definable, then $p\sub \dcl_{\cal N}(A)$.

The next fact can also be extracted  from \cite[Section 3]{pi}.

 \begin{fact}\label{fact-ei2}
Assume  $\cal N$ is an ordered pregeometric structure that satisfies (*) from the introduction. Then $\cal N$ eliminates imaginaries.
 \end{fact}
\begin{proof}
Let $X\sub N^n$ be a definable set. We need to show that $X$ has a smallest definably closed defining set. Assume $X$ is $B_0$-definable,  $B_0\sub N$ is finite and $\dcl_{\cal N}$-independent, and $|B_0|$ is least  possible. We claim that $B=\dcl_{\cal N}(B_0)$ is the smallest definably closed defining set for $X$. Suppose not. Then there is another set $C\sub N$ such that $B\not\sub \dcl_{\cal N}(C)$, and $X$ is $C$-definable. Let $A=B\cap \dcl_{\cal N}(C)$.  By (*), $X$ is $A$-definable. Moreover, $A\subsetneqq B$ and $A$ is $\dcl_{\cal N}$-closed. By the exchange property in pregeometric theories, $A=\dcl_{\cal N}(A_0)$, for some $\dcl_{\cal N}$-independent $A_0$ with size $|A_0|<|B_0|$. Then $X$ is also $A_0$-definable, contradicting the choice of $B_0$.
\end{proof}

\subsection{The induced structure}
Recall from the introduction that
$$P_{ind(D)}=\la P, \{R\cap P^l: R\sub M^l \text{ $\cal L_{D}$-definable }, l\in \N\}\ra.$$

\begin{remark}\label{rmk-P}
Assuming (ind)$_D$, $P_{ind(D)}$ is weakly o-minimal. Indeed, every set $X\sub P$ definable in $P_{ind(D)}$ is of the form $Y\cap P$, for an $\cal L$-definable $Y\sub M$, and hence a finite union of  convex subsets of $P$. This description need no longer be true in general.
For example, let  $\cal M$ be the real field and $P= 2^\Z 3^\Z$. Let $f:\R^{>0}\to \R^{>0}$ with $f(x)=\sqrt{x}$. Then the projection of $\Gamma(f) \cap P^2$ onto the first coordinate
is the set of all elements in $P$ which are divisible by $2$ in the multiplicative sense; that is, the set of all elements of the form $2^{2m} 3^{2n}$, $m, n\in \Z$. This set is  not a finite union of convex subsets of $P$, but it is definable in $P_{ind(\es)}$.
Thus $\la \cal M, P\ra$ satisfies (OP) (by \cite[Section 2]{egh}) and (dcl)$_\es$ (see Section \ref{sec-mann}), but not (ind)$_\es$.
\end{remark}

\begin{lemma}\label{extendf} Assume \textup{(OP)} and \textup{(ind)$_D$}.
Let $f:P^n\to P^k$ be an $A$-definable map in $P_{ind(D)}$. Then there is an $\cal L_{AD}$-definable map $F:M^n\to M^k$ that extends $f$.
\end{lemma}
\begin{proof}
By (ind)$_D$, there is an $\cal L_{AD}$-definable set $T\sub M^{n+k}$ such that $\Gamma(f)= T\cap P^{n+k}$. By (OP), $P$ is dense in an $\cal L$-definable set, and by o-minimality, we may assume that for every $x\in P^n$, $T_x$ is a singleton. The set
$$X=\{x\in \pi(T):\, T_x \text{ is a singleton}\}$$
is $\cal L_{AD}$-definable. So, $P^n\sub X$. Now let
$$T'=\left(\bigcup_{x\in X} \{x\}\times T_x\right)\cup \{(x, 0) : x\in M^n\sm X\}.$$
Then $T'$ is $\cal L_{AD}$-definable, it is the graph of a function $F: M^n\to M^k$, and
$\Gamma(f)=T'\cap P^{n+k}$, as required.
\end{proof}

We denote the definable closure operator in $P_{ind(D)}$ by $cl_D$.

\begin{cor}\label{clDA} Assume \textup{(OP)} and \textup{(ind)$_D$}. Then for every $A\sub P$, $\cl_D(A)=\dcl (AD)\cap P$.
\end{cor}
\begin{proof}
The inclusion $\supseteq$ is immediate from the definitions, whereas the inclusion $\sub$ is immediate from Lemma \ref{extendf}.
\end{proof}

By Corollary \ref{clDA} and the fact that $\dcl(-D)$ defines a pregeometry in $\cal M$, it follows easily that, under  \textup{(ind)$_D$}, $cl_D(-)$ defines a pregeometry in $P_{ind(D)}$. This pregeometry need not satisfy usual properties known for the definable closure in  o-minimal structures. For example, as pointed out in  \cite[Proposition 2.4]{ehk}, if $\WM$ is a dense pair, then there are  $A\sub P$, such that $cl_D(A)\not\elsub P_{ind(D)}$. It is natural to ask the following:

\begin{question}\label{qnD} Under what assumptions on $\WM$ and $D$,  is it true that for all $A\sub P$, $cl_D(A)\elsub P_{ind(D)}$?
\end{question}

\section{Proofs of Theorems A and B}\label{sec-main}

In this section we prove  elimination of imaginaries  for $P_{ind(D)}$ as in Theorem A, and  deduce Theorem B from it. Our goal is to establish (*) from Introduction  for $\cal N=P_{ind(D)}$ (Lemma \ref{pillay2} below). The strategy is to reduce its proof to \cite[Proposition 2.3]{pi}, which is an assertion of (*) for $\cal M$.

We need the following key technical lemma.  Right after it, we illustrate  the main points of its  proof with an example.

 \begin{lemma}\label{general1} Assume \textup{(OP)} and that $D$ is $\dcl$-independent over $P$.  Let $B, C\sub P$,  $X\sub P^n$, and $Y, Z\sub M^n$ such that $Y$ is $\cal L_{BD}$-definable, $Z$ is $\cal L_{CD}$-definable and
\begin{equation}
X=P^n\cap Y=P^n\cap Z.\label{YZ}
\end{equation}
Then there is $W\sub M^n$, both $\cal L_{BD}$-definable and $\cal L_{CD}$-definable, such that
$$X=P^n\cap W.$$
\end{lemma}
\begin{proof}

We work by induction on $\dim (Y\cup Z)$.  First note that, by (\ref{YZ}), $X$ is both $BD$-definable and $CD$-definable in $\la \cal M, P\ra$. Since $B, C\sub P$, by (OP) it follows that $\overline X$ is $\cal L_{BD}$-definable and $\cal L_{CD}$-definable.

If $\dim (Y\cup Z)=0$, then $X$ is finite, and hence $X=\overline X$. So $X$ is both $\cal L_{BD}$-definable and $\cal L_{CD}$-definable, and can let $W=\overline X$.

Assume $\dim(Y\cup Z)=k>0$. Let
$$K=\overline{(\overline X\sm (Y\cup Z))\cap P^n}.$$

\noindent\textbf{Claim 1.} \emph{$K$ is $\cal L_{BD}$-definable and $\cal L_{CD}$-definable.}
\begin{proof}[Proof of Claim 1.]
It suffices to prove that $$(\overline X\sm (Y\cup Z))\cap P^n=(\overline X\sm Y)\cap P^n=(\overline X\sm Z)\cap P^n.$$
since the second (respectively, third) part is $BD$-definable (respectively, $CD$-definable) in $\la \cal M, P\ra$ and hence its closure $\cal L_{BD}$-definable (respectively, $\cal L_{CD}$-definable). We prove the first equality, the other being completely analogous. We only need to prove $\supseteq$. Let $x\in (\overline X\sm Y)\cap P^n$. We claim that $x\not\in Z$. Indeed, if $x\in Z$, then $x\in Z\cap P^n=X$, and hence $x\in X\sm Y$, contradicting (\ref{YZ}).
\end{proof}

\noindent\textbf{Claim 2.} We have
$$\dim K < k.$$
\begin{proof}[Proof of Claim 2.]
By (\ref{YZ}), we have $\overline X\sub \overline{Y\cup Z}$. Therefore,
$$\overline X\sm (Y\cup Z)\sub \overline{(Y\cup Z)}\sm (Y\cup Z)$$
has dimension $<k$, and hence so does $K\sub \overline{\overline X\sm (Y\cup Z)}$.
\end{proof}

\noindent\textbf{Claim 3.} We have
$$ (\overline X \sm K)\cap  P^n\sub X.$$
\begin{proof}[Proof of Claim 3.]
Let $x\in (\overline X \sm K)\cap  P^n$. By (\ref{YZ}) it suffices to show that $x\in Y\cup Z$. Assume not. Then $x\in (\overline X\sm (Y\cup Z))\cap P^n\sub K$, a contradiction.
\end{proof}

By Claim 3  and since $X\sub P^n$, we can write
\begin{equation}
   X=(X\cap K)\cup (X \cap (\overline X\sm K))=(X\cap  K)\cup (P^n \cap (\overline X\sm K)).\label{W1}
 \end{equation}
We also have $X\cap K= P^n \cap Y\cap K=P^n\cap Z\cap K$. By Claim 1, $Y\cap K$ is $\cal L_{BD}$-definable and $Z\cap K$ is $\cal L_{CD}$-definable. By Claim 2, $\dim ((Y\cap K)\cup(Z\cap K))<k$. Hence, by inductive hypothesis, there is $W_1\sub M^n$, both $\cal L_{BD}$-definable and $\cal L_{CD}$-definable, such that
\begin{equation}
  X\cap K=P^n\cap W_1.\label{W2}
\end{equation}
Let also $W_2= \overline X\sm K$. Again by Claim 1, $W_2$ is $\cal L_{BD}$-definable and $\cal L_{CD}$-definable. Therefore, for $W=W_1\cup W_2$, (\ref{W1}) and (\ref{W2}) give
$$X=P^n\cap W,$$
as required.
\end{proof}

\begin{example} The above proof can be illustrated as follows. Let   $\widetilde{\cal M}=\la \widetilde \R, P\ra$ be a dense pair with $\widetilde \R$ the real field, $l_\beta$ and $l_\gamma$ two non-parallel lines in $\R^2$,
$$Y=\R^2\sm l_\beta\,\,\,\,\,\,\text{ and }\,\,\,\,\,\, Z=\R^2\sm l_\gamma,$$
and $X$, $K$, $W_1$ and $W_2$  the sets defined in the above statement and proof.   Assume $l_\beta$ is $\cal L_B$-definable and $l_\gamma$ is $\cal L_C$-definable. Let $D=\es$. Apart from the intersection point $c\in l_\beta \cap l_\gamma$, the two lines cannot contain any other element of $P^2$. Indeed, such an element would belong to only one of $Y\cap P^2$ and $Z\cap P^2$, contradicting (\ref{YZ}). It follows that $\{c\}$ is both $\cal L_B$-definable and $\cal L_C$-definable. There are two cases:\smallskip

\noindent Case I. $c\in P^2$. In this case, $K=\{c\}$, $W_2= \R^2\sm \{c\}=Y\cup Z$
and $$X=(\R^2\sm \{c\})\cap P^2.$$

\noindent Case II. $c\not\in P^2$. In this case, $K=\es$, $W_2=\R^2$ and
$$X=\R^2 \cap P^2.$$ Note that in the second case, even though we also have $X=(Y\cup Z)\cap P^2$, the set $Y\cup Z$ is neither $\cal L_B$-definable nor $\cal L_C$-definable. In both cases, $W_1=\es$ and
$X=(\overline X\sm K)\cap P^2.$
We leave it to the reader to construct examples on $\widetilde \R$ where the above cases actually occur.
\end{example}

\begin{lemma}\label{pillay2} Assume \textup{(OP)}, \textup{(dcl)}$_D$ and \textup{(ind)$_D$}  hold for $\widetilde{\cal M}$ and $D$, and that $D$ is $\dcl$-independent over $P$. Let $B, C\sub P$ and $A=cl_D(B) \cap cl_D(C)$. If $X\sub P^n$ is $B$-definable and $C$-definable in $P_{ind(D)}$, then $X$ is $A$-definable in $P_{ind(D)}$.
\end{lemma}
\begin{proof}
Let $X\sub  P^n$ be $B$-definable and $C$-definable in $P_{ind(D)}$. By (ind)$_D$,
$$X=P^n\cap Y = P^n \cap Z,$$
for some $\cal L_{BD}$-definable $Y\sub M^n$ and $\cal L_{CD}$-definable $Z\sub M^n$. By Lemma \ref{general1}, there is $W\sub M^n$, both $\cal L_{BD}$-definable and $\cal L_{CD}$-definable, such that $X=P^n\cap W.$ By \cite[Proposition 2.3]{pi}, $W$ is \cal L-definable over $\dcl(BD)\cap \dcl(CD)$. By (dcl)$_D$, $W$ is \cal L-definable over $\dcl(BD)\cap \dcl(CD)\cap PD$. Hence $X$ is definable over $\dcl(BD)\cap \dcl(CD)\cap P$ in $P_{ind(D)}$. But $$\dcl(BD)\cap \dcl(CD)\cap P=cl_D(B) \cap cl_D(C)=A,$$ and hence $X$ is $A$-definable in $P_{ind(D)}$.
\end{proof}

We can now conclude Theorems A and B.

\begin{proof}[Proof of Theorem A] By Fact \ref{fact-ei2} and Lemma \ref{pillay2}.
\end{proof}

\begin{proof}[Proof of Theorem B]  Let $h:M^m\to M^n$ be an $\cal L_D$-definable map such that $X\sub h(P^m)$. Consider the following  equivalence relation $E$ on $M^m$:
$$x E y \Lrarr h(x)=h(y).$$
Note that $E\cap (P^{m}\times P^m)$ is an  equivalence relation on $P^m$, which is $\emptyset$-definable in $P_{ind(D)}$. Since $P_{ind(D)}$ eliminates imaginaries, there is a $\emptyset$-definable in $P_{ind(D)}$ map $f:P^m\to P^l$, for some $l$, such that for every $x, y\in P^m$,
$$f(x)=f(y) \Lrarr x E y.$$
Define $\tau:X\to P^l$, given by $\tau(h(x))=f(x)$. Then $\tau$ is well-defined and injective. Since
$$\tau(y)=z \,\Lrarr\, \exists x\in P^m,\, h(x)=y \,\text{ and }\, f(x)=z,$$
it is also $D$-definable (in $\WM$).
\end{proof}

Finally, we turn to the proof of Corollary \ref{corD},  where the parameter set for $X$ is not required to be $\dcl$-independent over $P$. We need the following lemma.

\begin{lemma}\label{preserve}
  Assume \textup{(OP)}, \textup{(dcl)$_D$} and \textup{(ind)$_D$} hold for $\widetilde{\cal M}$ and $D$. Let $\cal M'$ be the expansion of $\cal M$ with constants for all elements in $P$, and $\WM'=\la \cal M', P\ra$. Then \textup{(OP)}, \textup{(dcl)}$_D$ and \textup{(ind)}$_D$ hold for $\WM'$ and $D$.
\end{lemma}
\begin{proof}
For (OP), let $V\sub M$ be  $A$-definable in $\WM'$. So $V$ is  $A P$-definable in $\WM$. If $A\sm P$ is $\dcl_{\cal M'}$-independent over $P$, then $AP\sm P=A\sm P$ is $\dcl$-independent over $P$, and hence it is $\dcl$-independent over $P$. By (OP) for $\WM$, $\overline V$ is $\cal L_{AP}$-definable. Hence it is $A$-definable in $\cal M'$.

For (dcl)$_D$ for $\cal M'$, observe that if $B\sub P$, then
$$\dcl_{\cal M'}(BD)=\dcl(PD).$$
Now let $B, C\sub P$ and
$$A=\dcl_{\cal M'}(BD)\cap \dcl_{\cal M'}(CD)\cap P.$$
 We have
$$\dcl_{\cal M'}(AD)=\dcl(PD)= \dcl_{\cal M'}(BD)\cap \dcl_{\cal M'}(CD),$$
as required.

For (ind)$_D$,
let $X\sub P^n$ be $A$-definable in the $D$-induced structure on $P$ by $\cal M'$. It follows that $X$ is $AD$-definable in $\WM'$. So it is $ADP$-definable in $\WM$. Hence $X$ is the trace of an $\cal L_{ADP}$-definable set $Y$. Such $Y$ is $AD$-definable in $\cal M'$, as required.
\end{proof}

\begin{proof}[Proof of Corollary \ref{corD}] Let $D$ be a maximal subset of $A$ which is $\dcl$-independent over $P$. Then $\dcl(A)\sub \dcl(DP)$. Let $\WM'$ be the expansion of $\WM$ as in Lemma \ref{preserve}. Hence (OP), (dcl)$_D$ and (ind)$_D$ hold for $\WM'$ and $D$. Moreover, $X$ is $D$-definable in $\WM'$. By Theorem B, there is an injective map $\tau:X\to P^l$, which is $D$-definable in $\WM'$. Hence $\tau$ is $DP$-definable (in $\WM$), and thus also $AP$-definable.
\end{proof}

\begin{remark} In Example \ref{noEI} below we show that the assumption of $D$ being $\dcl$-independent in Theorem A is necessary for $P_{ind(D)}$ to eliminate imaginaries. However, it is still possible to have $A$ not $\dcl$-independent over $P$, and yet,  $P_{ind(A)}$ eliminate imaginaries. This is the case whenever there are $D\sub A$, $\dcl$-independent over $P$, and $P_0\sub P$, such that
\begin{equation}
  \dcl(A)=\dcl(D P_0).\label{SCB}
\end{equation}
Indeed, if $E$ is a $\es$-definable equivalence relation in $P_{ind(A)}$, then it is also $\es$-definable in $P_{ind(D)}$. Let $f$ be as in Definition \ref{ei}, $\es$-definable in $P_{ind(D)}$. Then $f$ is $\es$-definable in $P_{ind(A)}$.

An example where assumption (\ref{SCB}) holds, for $A$ not $\dcl$-independent over $P$,  is when $\WM$ is a dense pair,  $d\not\in P$,  $A=\dcl(dP)$ and $D=\{d\}$.
\end{remark}

\subsection{On weakly o-minimal structures}\label{sec-wen} The reader may wonder why we do not directly apply or adopt elimination of imaginaries  results known for weakly o-minimal structures. Wencel \cite{wen} claims that a weakly o-minimal structure  $\cal N$ with   `strong cell decomposition' property (SCD), such that for every $A\sub N$, $\dcl_{\cal N}(A)\elsub \cal N$, eliminates imaginaries. One natural approach would be to assert those two assumptions for our $P_{ind(D)}$.
As pointed out right before Question \ref{qnD}, the latter property fails in the case of dense pairs, yet it might be true in other settings.
Even so, the proof in \cite{wen} appears to contain a serious gap. Namely, Theorem 6.3 in that reference is proved by imitating the proof of Pillay \cite[Proposition 3.2]{pi}, where at some point one needs to verify the following statement (as pointed out in our introduction).
\begin{itemize}
  \item[(**)]  Let $B, C\sub N$ and $N_0\elsub \cal N$ with $N_0=\dcl_{\cal N}(B) \cap \dcl_{\cal N}(C)$. If $X\sub N^n$ is $B$-definable and $C$-definable, then $X$ is $N_0$-definable.
\end{itemize}
To establish (**), the author uses Proposition 6.2, but it is unclear to us how to obtain its assumption, namely that every  convex subset of $\cal N$ which is $B$-definable and $C$-definable is also $N_0$-definable.
What one can extract from Wencel's account is the following fact, whose verification is left to the reader.

\begin{fact}[Wencel \cite{wen}]\label{fact-wencel}
Let \cal N be a weakly o-minimal structure with \textup{(SCD)}, and assume that:
\begin{enumerate}
  \item for every $B, C\sub N$
  with $N_0=\dcl_{\cal N}(B) \cap \dcl_{\cal N}(C)$,  every convex subset of $\cal N$ which is $B$-definable and $C$-definable is also $N_0$-definable.
  \item  for every $A\sub N$, $\dcl_{\cal N}(A)\elsub \cal N$.
\end{enumerate}
Then $\cal N$ admits elimination of imaginaries.
\end{fact}

\section{Proof of Theorem C}\label{sec-examples}

In this section, we establish the assumptions of Theorems A and B in three main examples of tame expansions of o-minimal structures by dense sets. Properties (OP) and (ind)$_D$ have already been considered in the literature and may as well be extractable from other resources, but we prove them anyway below in Fact \ref{inparticular}  in order to be complete. We first need the following general lemma.

\begin{lemma}\label{opind} Assume \textup{(OP)} and that $D$ is $\dcl$-independent over $P$. Suppose $X\sub P^n$ is $A$-definable in $P_{ind(D)}$, and it is the trace of an $\cal L$-definable set. Then $X$ is the trace of an $\cal L_{AD}$-definable set.
\end{lemma}
\begin{proof} Observe first that since $X$ is $A$-definable in $P_{ind(D)}$, it follows from the definition of $P_{ind(D)}$ that $X$ is $AD$-definable (in $\WM$). By (OP), since $AD\sm P\sub D$ is $\dcl$-independent over $P$, $\overline X$ is $\cal L_{AD}$-definable. Let $X=P^n\cap Y$, for some $\cal L$-definable $Y\sub M^n$. We do induction on $\dim Y$. If $\dim Y=0$, then $X$ is finite and hence $X=\overline X=P^n \cap \overline X$ is as needed. Now let $\dim Y>0$. The set
$$P^n\cap (\overline X\sm Y)= (P^n \cap \overline X) \sm (P^n \cap Y)= (P^n \cap \overline X) \sm X$$
is $A$-definable in $P_{ind(D)}$.  Since $\overline X\sm Y\sub \overline Y\sm Y$ has dimension $<\dim Y$, by inductive hypothesis there is an $\cal L_{AD}$-definable $W\sub M^n$, such that
$$P^n \cap (\overline X\sm Y)=P^n \cap W.$$
Since $\overline X\sm W$ is $\cal L_{AD}$-definable and
$$X=(P^n\cap \overline X)\sm (P^n \cap W)=P^n\cap (\overline X\sm W),$$
we are done.
\end{proof}

\begin{fact}\label{inparticular}
Let $\widetilde{\cal M}$ be any of the examples in Theorem C, and suppose  $D$ is $\dcl$-independent over $P$. Then \textup{(OP)} and \textup{(ind)}$_D$ hold.
\end{fact}
\begin{proof}
Property (OP) was shown for all three examples in \cite[Section 2]{egh} (and, in fact, without assuming divisibility of $P$ in the third example). For (ind)$_D$, in all three examples, every definable $X\sub P^n$ is the trace of an $\cal L$-definable set (\cite[Theorem 2]{vdd-dense}, \cite[2.16]{dms2}, \cite[Theorem 7.2]{dg}). By Lemma \ref{opind}, we are done.
\end{proof}

Thus the rest of this section is devoted to showing (dcl)$_D$.  Our strategy is to introduce yet another property for $\WM$ and $D$, prove that together with (OP) it implies (dcl)$_D$ (Proposition \ref{instead}), and then verify it in our examples.

Consider the following property for $\WM$ and $D$:

\begin{itemize}
  \item[($\dcl'$)$_D$]For every $\alpha\in \dcl(PD)$, there is $q\sub P$, such that
      $$\dcl(qD)=\dcl_{\cal L(P)}(\alpha D).$$
\end{itemize}

\begin{remark}
We could  have equally considered ($\dcl'$)$_D$ as one of the main assumptions in this paper, in place of (dcl)$_D$. We chose, however,  the latter because it involves only definability in \cal M.
\end{remark}

In Proposition  \ref{dclnec2} below we give a complete picture of several properties mentioned. For handling our examples in this section, we only need Proposition \ref{instead} below. First, a very useful fact.

\begin{fact}\label{op}
Suppose \textup{(OP)} holds, and $D$ is $\dcl$-independent over $P$. Then for every $A\sub P$, $\dcl_{\cal L(P)}(AD)=\dcl(AD)$.
\end{fact}
\begin{proof}
Take $x\in \dcl_{\cal L(P)}(AD)$. That is, the set $\{x\}$ is $AD$-definable in $\WM$. By (OP),  since $AD\sm P\sub D$ is $\dcl$-independent over $P$, we have that $\overline{\{x\}}$ is $\cal L_{AD}$-definable. But $\overline{\{x\}}=\{x\}$.
\end{proof}

\begin{prop}\label{instead}  Suppose $D$ is $\dcl$-independent over $P$. Then:
$$\text{\textup{(OP)} and \textup{($\dcl'$)$_D$} $\Rarr$ \textup{(dcl)$_D$.}}$$
\end{prop}
\begin{proof}
\noindent Let $B, C\sub P$ and
$$A=\dcl(BD)\cap \dcl(CD)\cap P.$$ We will show that $$\dcl(AD)=\dcl(BD)\cap \dcl(CD).$$

\noindent $(\sub)$. This part follows immediately from properties of $\dcl$. Indeed:
$$\dcl(AD) = \dcl(\dcl(BD)\cap \dcl(CD)\cap PD)\sub\dcl(\dcl(BD)\cap \dcl(CD))=$$ $$=\dcl(BD)\cap \dcl(CD).$$\vskip.2cm

\noindent $(\supseteq)$. Let $\alpha \in \dcl(BD)\cap \dcl(CD)$. By ($\dcl'$)$_D$, there is $q\sub P$ such that
$$\dcl(qD)=\dcl_{\cal L(P)}(\alpha D).$$
Observe that
$$q\sub\dcl_{\cal L(P)}(\alpha D)\sub \dcl_{\cal L(P)}(B D) \cap \dcl_{\cal L(P)}(C D)=\dcl(BD)\cap \dcl(CD),$$
by Fact \ref{op}. Hence $q\sub A$ and
$\alpha\in \dcl(qD)\sub \dcl(AD).$
\end{proof}

$ $\\
\textbf{In the rest of this section, we fix $D\sub M$  be $\dcl$-independent over $P$.}
We  proceed to prove (dcl$'$)$_D$ in our  examples.

\subsection{Dense pairs} Let $\WM =\la \cal M, P\ra$ be a dense pair. As shown in the Appendix, elimination of imaginaries for $P_{ind(D)}$ fails in general in this setting. We introduce a property of $\WM$ that implies (dcl$'$)$_D$ (Proposition \ref{LPdcl}), and then prove it for dense pairs of real closed fields (Proposition \ref{Pcond}). \\

\noindent\textbf{$\cal L_P$-extension property:}
Let $X\sub M^n$ be $\cal L$-definable and $A$-definable. Suppose $P^n$ is dense in $X$. Then $X$ is contained in a set $Y\sub M^n$, which is $\cal L_P$-definable and $A$-definable, with $\dim Y=\dim X$.

\begin{prop}\label{LPdcl}
Suppose $\WM$ has the $\cal L_P$-extension property. Then $(\dcl')_D$ holds.
\end{prop}

\begin{proof}
Let $\alpha\in \dcl(PD)$.  If $\alpha\in \dcl(D)$, then the empty tuple $q=\es$ verifies ($\dcl'$)$_D$. Indeed, we have
$$\dcl_{\cal L(P)}(\alpha D)= \dcl_{\cal L(P)}(D)=\dcl(D),$$ by Fact \ref{op}. So we may assume that $\alpha\not\in \dcl(D)$. Then there is an $\cal L_{D}$-definable map $f:M^n\to M$,  and a  tuple $b=(b_1, \dots, b_n)\in P^n$, such that $$\alpha=f(b).$$
Let $S=f^{-1}(\alpha)$ and $T=\overline {S\cap P^n}$. Let $k=\dim T$. Observe that both $S$ and $P^n$ are dense in $T$, and that $T$ is  $\alpha D$-definable (but not necessarily $\cal L_{\alpha D}$-definable).
By the $\cal L_P$-extension property, for $X=T$ and $A=\alpha D$, $T$ is contained in a set $Y\sub M^n$, which is $\cal L_P$-definable and $\alpha D$-definable, and has dimension $k$. Let $q$ be the canonical parameter of $Y$ in the sense of \cal M. Since $Y$ is $\cal L_P$-definable, $q\sub \dcl(P)=P$. Since $Y$ is $\alpha D$-definable, we also have $q\sub \dcl_{\cal L(P)}(\alpha D)$. Indeed, let $\tau$ be an automorphism of $\WM$ that fixes $\alpha D$ pointwise. Since $Y$ is $\alpha D$-definable, $\tau(Y)=Y$. But $\tau$ is also an automorphism of $\cal M$ only. Since $q$ is a canonical parameter for $Y$, $\tau(q)=q$. Hence $q\sub \dcl_{\cal L(P)} (\alpha D)$.

Now let $C$ be the set of all points in $Y$ for which there is an open box $V\sub M^n$ containing them with $\dim V\cap Y=k$ and $f_{\res V\cap Y}$ constant. Then $C$ is $\cal L_{qD}$-definable and $\alpha D$-definable. Moreover, by o-minimality, $f(C)$ is finite.
Also, since  $S$ is dense in $T$, we have $\dim T\cap S=k$. Let $S'= T\cap S$. Since $S'\sub T\sub Y$ and $\dim S'=\dim Y$, we can easily find an open box $V\sub M^n$ with $V\cap S'=V\cap Y$ and $\dim V\cap Y=k$. (To see this, take a cell decomposition $\cal C$ of $M^n$ that partitions both $S'$ and $Y$, and let $E\in \cal C$ be a $k$-cell contained in $S'$. Then every $c\in E$ is contained in such a box $V$.) Since $S'\sub S$ and $f(S)=\{\alpha\}$, we obtain that $f (V\cap Y)=\{\alpha\}$, and hence $\alpha\in f(C)$.
Since $f(C)$ is finite and $\cal L_{qD}$-definable, clearly,  $\alpha\in \dcl(qD)$, as required.
\end{proof}

We will use the following basic fact from linear algebra.
\begin{fact}\label{fact-linear}
Let $\cal S\sub M^{r+1}$ be a linear subspace of dimension $k<r+1$. Then, after perhaps permuting coordinates, there are $c_1, \dots, c_k\in M$, such that for every $w=(w_1, \dots, w_{r+1})\in \cal S$,
$$w_{k+1}= c_1 w_1 +\dots +c_k w_k.$$
\end{fact}

We will also need the following fact.
\begin{fact}\label{choosey}
Suppose $\es\ne X\sub M^n$ is $\cal L$-definable and $A$-definable. Then there is $y\in X$ contained in $\dcl_{\cal L(P)}(A)$.
\end{fact}
\begin{proof}
It is straightforward to check that the tuple $e(X)\in X$  defined in \cite[Chapter 6 (1.1)]{vdd-book} is contained in $\dcl_{\cal L(P)}(A)$.
\end{proof}

We are now ready to conclude the main result of this subsection.

\begin{prop}\label{Pcond}
The $\cal L_P$-extension property holds for dense pairs of real closed fields.
\end{prop}
\begin{proof}  If $\dim X = n$, let $Y=M^n$. Assume $\dim X< n$, and $X\ne \es$. By quantifier elimination for real closed fields, there is a non-zero polynomial
$Q(x)\in M[x]$, where $x$ is an $n$-tuple, whose zero set, call it $Q$, contains $X$. For $i=(i_1,  \dots, i_n)\in \N^n$,  denote $$x^{i}=x_1^{i_1} \cdot \dots\cdot x_n^{i_n}.$$
Let us write
\begin{equation}
  Q(x)= s_1  x^{j_1} +\dots + s_r x^{j_r}+s_{r+1} x^{j_{r+1}},\label{eqQ}
\end{equation}
for some $s_m\in M$ and suitable indices $j_m$, $m=1, \dots, r+1$, and denote by $f: X\to M^{r+1}$ the $\cal L_\es$-definable continuous map
$$(x_1, \dots, x_n)\mapsto (x^{j_1}, \dots,  x^{j_{r+1}}).$$
We may assume that $r$ is minimal such; namely, for every  non-zero polynomial $Q'(x)$ whose zero set contains $X$, the corresponding expression (\ref{eqQ}) has $r'$ many summands, for some $r'\ge r+1$.

Let $Z$ be the $r$-dimensional linear subspace of $M^{r+1}$ defined by
\begin{equation}
  s_1  w_1 +\dots +  s_r w_r +s_{r+1} w_{r+1}=0.\label{eqsr}
\end{equation}
Since $X\sub Q$, we have $f(X)\sub Z$.

$ $\\
\noindent\textbf{Claim.} {\em There are $y_1, \dots, y_r\in X$ such that $f(y_1), \dots, f(y_r)$ are linearly independent. }

\begin{proof}[Proof of Claim]
Assume not. Then $f(X)\sub M^{r+1}$ is contained in a $k$-dimensional subspace, for some $k<r$. (The linear span of any set $Y$ (here, $Y=f(X)$) has dimension equal to the maximum $k$ such that $Y$ contains $k$ independent elements.)  By Fact \ref{fact-linear}, after perhaps permuting coordinates, there are $c_1, \dots, c_k\in M$, such that for every $w=(w_1, \dots, w_{r+1})\in f(X)$,
$$w_{k+1}= c_1 w_1 +\dots +c_k w_k.$$
In particular, since $f(X)\sub Z$, we obtain from (\ref{eqsr}) that for every such $w$,
$$ (s_1  + s_{k+1} c_1) w_1 +\dots + (s_k + s_{k+1} c_k) w_k+ s_{k+2} w_{k+2} +\dots + s_{r+1} w_{r+1}=0.$$
Consider the polynomial
$$Q'(x)= (s_1 +s_{k+1}  c_1) x^{j_1}+\dots + (s_k+s_{k+1} c_k) x^{j_k} +s_{k+2} x^{j_{k+2}} +\dots + s_{r+1} x^{j_{r+1}}.$$
Then the zero set of $Q'(x)$ contains $X$ and the above expression has $r<r+1$ summands. By minimality of $r$,  $Q'(x)$ must be the zero polynomial. But then the polynomial
$$Q(x) - Q'(x)= - (c_1 s_{k+1}  x^{j_1}+\dots +c_k s_{k+1}  x^{j_k}) +s_{k+1} x^{j_{k+1}}$$
has the same solution set as $Q(x)$, contradicting again the minimality of $r$.
\end{proof}

Dividing (5) by $s_{r+1}$, we may assume that $s_{r+1}=1$. We now show that  $s_m\in \dcl_{\cal L(P)}(A)\cap P$, for each $m=1, \dots, r$. By the claim, the set $Y$ of all $(y_1, \dots, y_r)\in X^r$ such that $f(y_1), \dots, f(y_r)$ are linearly independent is non-empty. Moreover, it is $A$-definable, since $X$ is. By Fact \ref{choosey}, we can choose $(y_1, \dots, y_r)\in Y$ with each $y_i\sub \dcl_{\cal L(P)}(A)$. We can thus form a linear system of $r$ equations
\begin{equation}
  s_1  f(y_i)_1 +\dots +  s_r  f(y_i)_r+  f(y_i)_{r+1}=0, \,\, i=1, \dots, r,\notag
\end{equation}
with coefficients in $\dcl_{\cal L(P)}(A)$ and unique solution $s_1, \dots, s_r$. This means that each $s_m\in \dcl_{\cal L(P)}(A)$.

Clearly, we can further find open boxes $W_i\sub M^{r+1}$  containing $f(y_i)$, $i=1, \dots, r$,  such that any $t_1, \dots, t_r$ with $t_i\in W_i$, are still linearly independent. By continuity of $f$, there are open boxes $B_i\sub M^n$ containing $y_i$, $i=1, \dots, r$, respectively, such that $f(B_i)\sub W_i$. Since $P^n$ is dense in $X$, there are $z_i \in P^n \cap X \cap B_i$. We can thus form a linear system of $r$ equations
\begin{equation}
  s_1  f(z_i)_1 +\dots +  s_r  f(z_i)_r+  f(z_i)_{r+1}=0, \,\, i=1, \dots, r,\notag
\end{equation}
with coefficients in $P$ and unique solution $s_1, \dots, s_r$. This means that each $s_m\in \dcl(P)=P$.
\end{proof}

\subsection{Expansions of $\cal M$ by a dense independent set}
Here we assume that $P\sub M$ is a dense $\dcl$-independent set. Following \cite[page 58]{vdd-book} and \cite[1.5]{dms}, we call a set $X\sub M^n$  \emph{regular} if it is convex in each coordinate, and \emph{strongly regular} if it is regular and all points in $X$ have pairwise distinct coordinates. A map $f:X\sub M^n\to M$ is called \emph{regular} if $X$ is regular and $f$ is, in each coordinate, either constant or strictly monotone and continuous.

\begin{proof}[\textbf{Proof of $(\dcl')_D$.}]
Let $\alpha\in \dcl(bD)$, with $b\in P^n$ and $n$ least possible. In particular, $b$ is $\dcl$-independent over $D$. We prove that $q=b$  verifies ($\dcl'$)$_D$.
Write $b=(b_1, b_2)$ where $b_1\in P^{n-1}$. There is an $\cal L_{D}$-definable map $f:M^n\to M$ with $f(b)=\alpha$. By \cite[1.6]{dms2}, there is a cell decomposition $M^n=\bigcup_i C_i$ into $\cal L_{D}$-definable cells, such that, for each open $C_i$,
\begin{itemize}
  \item each open $C_i$ is strongly regular, and
  \item for each open $C_i$, $f_{\res C_i}$ is regular.
\end{itemize}
Since $b$ is $\dcl$-independent over $D$, it must belong to an open, and hence strongly regular, $C_i$. If $f_{\res C_i}$ were constant in some coordinate, say the last one, then
$\alpha\in \dcl(b_1 D)$, contradicting the assumption on $n$. So $f$  is non-constant in each coordinate.
By \cite[1.8]{dms2}, $f^{-1}(\alpha)\cap P^n$ is finite, so
$$b\in \dcl_{\cal L(P)}(\alpha D).$$
Hence $\dcl_{\cal L(P)}(bD)=\dcl_{\cal L(P)}(\alpha D),$ and
$$\dcl(bD)=\dcl_{\cal L(P)}(\alpha D),$$
by Fact \ref{op},
as required.
\end{proof}

\subsection{Expansions of $\cal M$ by a dense multiplicative group with the Mann property}\label{sec-mann}
Let $\cal M=\la M, <, +, \cdot, 0, 1\ra$ be a real closed field. Let $G$ be a dense subgroup of $\la M^{>0}, \cdot\ra$. For every $a_1, \dots, a_r\in M$, a solution $(q_1, \dots, q_r)$ to the equation
$$a_1 x_1 + \dots +a_r x_r=1$$
is called \emph{non-degenerate} if for every non-empty $I\sub \{1, \dots, r\}$, $\sum_{i\in I} a_i q_i\ne 0$. We say that $G$  has the \emph{Mann property}, if  for every $a_1, \dots, a_r\in M$, the above equation
has only finitely many non-degenerate solutions  $(q_1, \dots, q_r)$ in $G^r$. Observe that the original definition only involved equations with coefficients $a_i$ in the prime field of $\cal M$, but, by \cite[Proposition 5.6]{dg}, the two definitions are equivalent.

We now assume that $P$ is a dense  subgroup of $\la M^{>0}, \cdot\ra$ with the Mann property, and work in $\WM=\la \cal M, P\ra$. Note that we do not assume divisibility of $P$ here.

\begin{proof}[\textbf{Proof of $(\dcl')_D$.}]
 Let $\alpha\in \dcl(PD)$.
Denote $K=\dcl(D)$.
Then there is a polynomial
$Q(x,y)\in K[x, y]$, where $x$ is an $n$-tuple, and $b_1, \dots, b_n\in P$, such that \begin{equation}
  Q(b_1, \dots, b_n, \alpha)=0,\label{eqMann}
\end{equation}
and such that $Q(b_1, \dots, b_n, y)$ is not the zero-polynomial.
Let $b=(b_1, \dots, b_n)$ and, for $i=(i_1,  \dots, i_n)\in \N^n$,  denote $$b^{i}=b_1^{i_1} \cdot \dots\cdot b_n^{i_n}.$$ Let us also write
$$Q(b_1, \dots, b_n, \alpha)= d_1 b^{j_1} \alpha^{k_1} +\dots + d_r b^{j_r} \alpha^{k_r},$$
for some $d_m\in \dcl(D)$ and suitable indices $j_m$ and $k_m$, $m=1, \dots, r$. We may assume that no sub-sum of the above expression is $0$, or else replace $Q(b_1, \dots, b_n, \alpha)$ by that sub-sum.
Now, divide equation (\ref{eqMann}) by  $d_r b^{j_r} \alpha^{k_r}$ to obtain an equation
\begin{equation}
q_1 a_1 +\dots + q_{r-1} a_{r-1}=1,\label{eqMann2}
\end{equation}
where $$q_m=\frac{b^{j_m}}{b^{j_r}}\in P$$ and
$$a_m=-\frac{d_m}{d_r}\alpha^{k_m-k_r}\in \dcl(\alpha D),$$ for $m=1, \dots, r-1$. Let $q=(q_1, \dots, q_{r-1})$. Equation (\ref{eqMann2}) still has the property that no sub-sum of the expression on the left is $0$. In other words,  $q$ is a non-degenerate solution to
$$x_1 a_1 + \dots + x_{r-1} a_{r-1}=1.$$
By Mann property, the last equation has only finitely many non-degenerate solutions
in $P$, and, since being a non-degenerate solution to that equation in $P$ is an $\alpha D$-definable property, we obtain that each $q_m\in\dcl_{\cal L(P)}(\alpha D)$. Hence $q\in \dcl_{\cal L(P)}(\alpha D)$.

Moreover, multiplying (\ref{eqMann2}) by $\alpha^{k_r}$, we obtain that $\alpha$ is solution to a polynomial equation with coefficients in $\dcl(qD)$. Hence $\alpha\in \dcl(qD)$.
By Fact \ref{op}, it follows that
$$\dcl_{\cal L(P)}(\alpha D)\sub\dcl_{\cal L(P)}(qD)=\dcl(q D),$$
as required.
\end{proof}

\begin{remark} Together with Fact \ref{inparticular} and Theorem A, we have established that in all examples of Theorem C, $P_{ind(D)}$ eliminates imaginaries. Note that the divisibility assumption in the third example was only used to guarantee (ind)$_D$ in the proof of Fact \ref{inparticular}, whereas (OP) and (dcl)$_D$ hold without it. It remains open whether elimination of imaginaries also hold without it.
\end{remark}

\begin{question}\label{qn-mann}
Let $\WM=\la \cal M, P\ra$ be an expansion of a real closed field $\cal M$ by a dense multiplicative subgroup $P$ of $\la M^{>0}, \cdot\ra$ with the Mann property. Suppose that  $D$ is $\dcl$-independent over $P$. Does $P_{ind(D)}$ eliminate imaginaries?
\end{question}

 \section{Optimality}\label{sec-optimality}
In this section, we establish that our results are optimal in three ways:
\begin{itemize}
  \item (Example \ref{noEI}). If $D$ is not $\dcl$-independent over $P$, then $P_{ind(D)}$ need not eliminate imaginaries.

\item (Proposition \ref{dclnec2}). Assume (OP) and (ind)$_D$, and let $D$ be $\dcl$-independent over $P$. Then:
$$\text{ $P_{ind(D)}$ eliminates imaginaries} \,\,\, \Lrarr\,\,\,\, \text{(dcl)}_D \,\,\,\,\Lrarr\,\,\,\, \text{($\dcl'$)}_D.$$

\item (Example \ref{dclDfail}). If we do not assume (OP), the above three properties need not hold.

  \end{itemize}

\begin{example}\label{noEI}
We give an example of $\widetilde \CM$ and $D$ where $D$ is not $\dcl$-independent over $P$ and $P_{ind(D)}$ does not eliminate imaginaries. Let $\CM$ be any o-minimal expansion of a real closed field, and $P\sub M$ any set such that there are $b_1, b_2, c_1, c_2\in P$ and $e\in M$, with
\begin{enumerate}
  \item $\{b_1, b_2, c_1, c_2, e\}$  $\dcl$-independent, and
  \item $e\not\in \dcl(P)$.
\end{enumerate}
(Such an $\CM$ can be a dense pair, an expansion of $\cal M$ by a dense independent set, or an expansion of $\cal M$ by a dense multiplicative group with the Mann property-- we will not use any further properties than the above two.)
Let $d\in M$ be defined by
$$b_1 + b_2 d = c_1 + c_2 e,$$
and $D=\{d, e\}$. Clearly, $D$ is not $\dcl$-independent over $P$. By (2), $d\not\in \dcl(P)$.
Moreover, none of $b_1, b_2$ is in $\dcl(c_1, c_2, d, e)$, since otherwise both would be in it and hence
$\rk(b_1, b_2, c_1, c_2, d, e)=4$, a contradicting (1).

Observe also that $(b_1, b_2)$ is the unique solution in $P^2$ to the equation
$$x_1 + x_2 d = c_1 +c_2 e.$$
Indeed, if $(e_1, e_2)\in P^2$ was another one, then $b_1 +b_2 d=e_1+e_2 d$, yielding $d\in \dcl(P)$, a contradiction. So the set
$$X=\{b_1, b_2\}=\{(x_1, x_2)\in P^2 : x_1 + x_2 d = c_1 +c_2 e\}$$
is both $\{b_1, b_2\}$-definable and $\{c_1, c_2\}$-definable in $P_{ind(D)}$.
We claim that $X$ cannot have a smallest $cl_D$-closed defining set $A$. Indeed, such  a set would have to be contained in $cl_D(b_1, b_2)\cap cl_D(c_1, c_2)$. However, $$b_i\not\in \dcl(c_1, c_2, d, e)\cap P=cl_D(c_1, c_2).$$
Hence $b_i\not\in A$, which is a contradiction, since $X=\{b_1, b_2\}$ is $A$-definable in $P_{ind(D)}$.
 \end{example}

Our next goal is Proposition \ref{dclnec2} below. First, a general statement.

\begin{lemma}\label{EIanyD} Suppose $P_{ind(D)}$ eliminates imaginaries. Then for every $\alpha\in \dcl(PD)$, there is $q\sub P$, such that
      $$\dcl_{\cal L(P)}(qD)=\dcl_{\cal L(P)}(\alpha D).$$
\end{lemma}

\begin{proof}
Let $\alpha\in \dcl(PD)$. So there are $b\sub P^n$ and an $\cal L_D$-definable map $f:M^n\to M$, such that $\alpha=f(b)$. The set
$$X=f^{-1}(\alpha)\cap P^n=\{x\in P^n: f(x)=f(b)\}$$
is $b$-definable in $P_{ind(D)}$. Let $q\sub P$ be a canonical parameter for it in the sense of $P_{ind(D)}$.  Now let $\tau$ be an automorphism of $\WM$ that fixes $D$ pointwise.
Observe that $X$ is also $\alpha D$-definable (in $\WM$). Hence we have
$$\tau(\alpha)=\alpha\,\Lrarr\, \tau (X)=X\,\,\Lrarr\,\, \tau (q)=q,$$
showing that
$$\dcl_{\cal L(P)}(qD)=\dcl_{\cal L(P)}(\alpha D).$$
\end{proof}

\begin{cor}\label{dclnec} Suppose $D$ is $\dcl$-independent over $P$. Then:
$$\text{ \textup{(OP)} and $P_{ind(D)}$ eliminates imaginaries} \Rarr  \text{\textup{($\dcl'$)}$_D$.}$$
\end{cor}

\begin{proof} Let $q\sub P$ be as in Lemma \ref{EIanyD}. By Fact \ref{op}, we have $\dcl(qD)=\dcl_{\cal L(P)}(qD)$. By Lemma \ref{EIanyD}, the result follows.
\end{proof}

\begin{prop}\label{dclnec2} Assume \textup{(OP)} and \textup{(ind)$_D$}, and let $D$ be $\dcl$-independent over $P$. Then
$$
\text{ $P_{ind(D)}$ eliminates imaginaries} \,\,\, \Lrarr\,\,\,\, \textup{(dcl)}_D \,\,\,\,\Lrarr\,\,\,\, \textup{($\dcl'$)}_D.
$$
\end{prop}
\begin{proof}
By Theorem A, Proposition \ref{instead} and Corollary \ref{dclnec}.
\end{proof}

We finally show that if we do not assume (OP), the above three properties need not hold. We do not know whether they hold, if we assume (OP) and (ind)$_D$ (but not that $D$ is $\dcl$-independent over $P$).

\begin{example}\label{dclDfail} Let $\cal M$ be our fixed o-minimal expansion of an ordered group, $p_1, p_2$  two $\dcl$-independent elements,  $P=\{p_1\}$ and $\alpha=p_1+p_2$. Then ($\dcl'$)$_\es$ fails for this $\alpha$, since there is no $q\in P$ such that $\alpha\in\dcl(q)$. Of course, (OP) also fails for $\WM=\la \cal M, P\ra$. Indeed, $\{p_2\}$ is $\dcl$-independent over $P$, but $\alpha\in \dcl_{\cal L(P)}(p_2)\sm \dcl(p_2)$, and hence, by Fact \ref{op}, (OP) fails.
\end{example}

\section{Appendix - An o-minimal trace that does not eliminate imaginaries\\
(with Philipp Hieronymi)}

We give an example of a dense pair $\la \cal M, P\ra$ of o-minimal structures, and $D\sub M$, such that $P_{ind(D)}$ does not eliminate imaginaries (equivalently, by Proposition \ref{dclnec2}, ($\dcl'$)$_D$ fails). Let $\cal R=\la \R, <, +,  1, x\mapsto \pi x_{\res [0, 1]}\ra$. Denote by $\cal L$ the language of $\cal R$ and by $\dcl$ the corresponding definable closure. It is well-known that $\Cal R$ does not define unrestricted multiplication by $\pi$. (For example, consider the isomorphism $x\mapsto e^x$ between $\cal R$ and the structure $\la \R_{>0}, <, \cdot, 1, x\mapsto {x^{\pi}}_{\res (1, e)}\ra$. By \cite{vdd-nondef}, the second structure does not define the unrestricted $x\mapsto x^\pi$, and hence the claim follows.)
 Moreover,  $\dcl(\emptyset)=\Q(\pi)$. Indeed, since $\pi$ is $\Cal L_{\emptyset}$-definable, it is easy to see that $\Q(\pi)\subseteq \dcl(\emptyset)$. Note that $\Q(\pi)$ is a $\Q(\pi)$-vector space and therefore a model of the theory of the structure $\cal R'=\la \R, <, +,  1, x\mapsto \pi x\ra$, which expands $\cal R$. Thus $\dcl(\es)\sub \dcl_{\cal R'}(\emptyset)\subseteq \Q(\pi)$.\newline

\noindent Let $\la \Cal M, P\ra$ be a saturated, elementary extension of $\la \cal R, \Q(\pi)\ra$. By compactness and the fact that $P$ is small, we can easily find $a_1,a_2\in M$ such that

\begin{enumerate}
\item $a_1 > n$ for every natural number $n$,
\item $q_1 a_1 < a_2  < q_2 a_1$ for all $q_1,q_2\in\Q$ with $q_1 < \pi < q_2$,
\item $a_1,a_2$ are $\Q$-linearly independent over $P$,
\item there are $p_1,p_2 \in P$ such that $a_1-1\leq p_1 \leq a_1$ and
\[
 \pi (a_1 - p_1) = a_2-p_2.
\]
\end{enumerate}
We show that ($\dcl'$)$_D$ fails for $\alpha=a_2$ and $D=\{a_1\}$.
By (4), $a_2\in \dcl(a_1 P)$. To disprove ($\dcl'$)$_{a_1}$, we need to prove that for every $b\sub \dcl_{\cal L(P)}(a_1 a_2)\cap P$, we have $a_2\not\in \dcl (b a_1)$. This will follow from the next three claims.

\begin{claim} Let $b \in \dcl_{\Cal L(P)}(a_1 a_2)\cap P$. Then $b\in \dcl (a_1 a_2)$.
\end{claim}
\begin{proof}
Let $b \in \dcl_{\Cal L(P)}(a_1 a_2)\cap P$.
Then there is an $\Cal L(P)$-formula $\psi$ such that
\begin{equation}\label{eq:lem}\tag{$\ast$}
\la \Cal M, P \ra \models \psi(b,a_1,a_2) \wedge \forall x (\psi(x,a_1,a_2) \leftrightarrow x=b).
\end{equation}
By the back and forth system for dense pairs, there is an $\Cal L$-formula $\varphi$ such that for all $c\in M$ $\dcl$-independent over $P$, and $p_0,p_1,p_2\in P$
\[
\la \Cal M, P\ra \models \psi(p_0,c,\pi(c-p_1)+p_2) \leftrightarrow \varphi(c,p_0,p_1,p_2).
\]
Then
\begin{align*}
\la \Cal M, P \ra \models \psi(x,a_1,a_2) &\leftrightarrow \exists p_1,p_2 \in P \  \pi (a_1 - p_1) = a_2-p_2 \wedge \psi(x,a_1,\pi(a_1-p_1)+p_2)\\
&\leftrightarrow \exists p_1,p_2 \in P \  \pi (a_1 - p_1) = a_2-p_2 \wedge \varphi(a_1,x,p_1,p_2)\\
& \leftrightarrow  \exists p_1,p_2 \in P \  \pi (a_1 - p_1) = a_2-p_2 \wedge \varphi(a_1,x,p_1,a_2-\pi(a_1 - p_1)).
\end{align*}
Set
\[
\chi(x_1,x_2,x_3,x_4) := \varphi(x_1,x_3,x_4,x_2-\pi(x_1 - x_3)).
\]
Therefore
\[
\la \Cal M, P \ra \models \psi(x,a_1,a_2) \leftrightarrow \exists p_1,p_2 \in P^2 \  \pi (a_1 - p_1) = a_2-p_2 \wedge \chi(a_1,a_2,x,p_1).
\]
Towards a contradiction, suppose that $b\notin\dcl(a_1 a_2)$. We can assume that for all $c\in M$
\[
\{ z \in M \ : \ \Cal M\models \chi(a_1,a_2,z,c)\}
\]
is an open interval or empty. By definable Skolem functions, we can assume that there is an $\Cal L$-$\emptyset$-definable function $f: M^3 \to M$ such that
\[
\Cal M \models \chi(x_1,x_2,x_3,x_4) \leftrightarrow x_3 = f(x_1,x_2,x_4).
\]
We now claim that $f(a_1,a_2,-)$ is constant on $[a_1-1,a_1]$. Suppose not. Then there is a subinterval $I\subseteq [a_1-1,a_1]$ such that $f(a_1,a_2,-)$ is injective on that interval. Since the set
\[
\{ p_1 \in P\cap I \ : \exists p_2 \in P \ \pi (a_1 - p_1) = a_2-p_2\}
\]
is dense in this interval, the set
\[
\{ z \in M \ : \  \exists p_1,p_2 \in P^2 \  \pi (a_1 - p_1) = a_2-p_2 \wedge z = f(a_1,a_2,p_1) \}
\]
is infinite, contradicting \eqref{eq:lem}. Thus $f(a_1,a_2,-)$ is constant on $[a_1-1,a_1]$. But then $b\in \dcl(a_1,a_2)$. Contradiction.
\end{proof}

\begin{claim}\label{claim2} Let $b \in \dcl(a_1 a_2)\cap P$. Then $b\in \Q(\pi)$.
\end{claim}
\begin{proof}
Let $q_1,q_2 \in \Q$, not both zero, say $q_2>0$. We first observe that $|q_1 a_1 +  q_2 a_2| >n$ for all $n\in \N$. Indeed, suppose not. Then there is $n\in \Z$ such that $0 \leq q_1 a_1 + q_2 a_2 - n <1$. Thus
\[
\frac{-q_1 a_1 +n}{q_2} \leq a_2 < \frac{-q_1 a_1 +(n+1)}{q_2}.
\]
However, this contradicts (1) or (2).\newline
Now let $f : M^2 \to M$ be $\Cal L$-$\emptyset$-definable such that $f(a_1,a_2) = b$. Since $\Cal L$-definable functions are piecewise given by terms and since $q_1 a_1 +  q_2 a_2 +c \notin [0,1]$ for all $c\in \Q(\pi)$, we can find $r_1,r_2 \in \Q$ and $d\in \Q(\pi)$ such that
\[
f(a_1,a_2)  = d + r_1a_1 +r_2 a_2.
\]
However, by (3) and since $b=f(a_1,a_2)$, we get that $r_1=r_2=0$. Thus $b=d\in \Q(\pi)$.
\end{proof}

\begin{claim} $a_2\not\in \dcl(\Q(\pi) a_1)$.
\end{claim}
\begin{proof} Suppose  $a_2\in \dcl(\Q(\pi) a_1)$.
By an argument as in the proof of Claim \ref{claim2}, there is $c\in \Q(\pi)$ and $q\in \Q$ such that $a_2 = c + q a_1.$ This contradicts (3).
\end{proof}

\end{document}